\documentclass[12pt]{article}
\usepackage{bbm}
\usepackage{latexsym}
\usepackage{mathrsfs}
\usepackage{amsmath}
\usepackage{graphicx}
\usepackage{amssymb}
\usepackage{color}
\usepackage{amsthm}

\newtheorem{thm}{Theorem}[section]
\newtheorem{Lemma}[thm]{Lemma}
\newtheorem{cor}[thm]{Corollary}

\newtheorem{defn}[thm]{Definition}

\numberwithin{equation}{section}

\usepackage{latexsym, bm}
\makeatletter
\long\def\@makecaption#1#2{%
 \vskip\abovecaptionskip
  \sbox\@tempboxa{{#1.}\quad #2}%
   \ifdim \wd\@tempboxa >\hsize
    { #1.}\quad #2\par
     \else
  \global \@minipagefalse
   \hb@xt@\hsize{\hfil\box\@tempboxa\hfil}%
   \fi
   \vskip\belowcaptionskip}
\makeatother

\usepackage{latexsym, bm}
\title{\vspace{0cm}{Classification of 2-extendable bipartite  and cubic non-bipartite vertex-transitive graphs}\footnote{Contract grant sponsor: NSFC; contract grant numbers: 11401279,  11371180 and 11201201; contract grant sponsor: the Specialized Research Fund
 for the Doctoral Program of Higher Education; contract grant number: 20130211120008; contract grant sponsor: Fundamental Research Funds for the Central Universities; contract grant number: lzujbky-2016-102.}}

\author{Qiuli Li\footnote{The
corresponding author.}  and Xing Gao\\
\small{School of Mathematics and Statistics, Lanzhou University, Lanzhou, Gansu 730000, China} \\
\small{E-mail addresses:  qlli@lzu.edu.cn and gaoxing@lzu.edu.cn}}

\date{}

\begin{document}
\maketitle
\begin{abstract}
In \cite{Chan95}, the authors classified the
2-extendable abelian Cayley graphs and posed the problem of
characterizing all 2-extendable Cayley graphs. We first show that a connected  bipartite Cayley (vertex-transitive) graph is 2-extendable if and only if it is not a cycle. 
It is known that a non-bipartite Cayley (vertex-transitive) graph  is 2-extendable when  it is of minimum degree at least five \cite{sun}. 
We next classify all 2-extendable cubic non-bipartite Cayley graphs and obtain that: 
a cubic  non-bipartite  Cayley graph with girth $g$ is
2-extendable if and only if $g\geq 4$ and it doesn't isomorphic to  $Z_{4n}(1,4n-1,2n)$ or $Z_{4n+2}(2,4n,2n+1)$ with $n\geq 2$. Indeed, we prove a more stronger result that  a cubic non-bipartite vertex-transitive graph with girth $g$ is
2-extendable if and only if $g\geq 4$ and it doesn't isomorphic to  $Z_{4n}(1,4n-1,2n)$ or $Z_{4n+2}(2,4n,2n+1)$ with $n\geq 2$ or the Petersen graph.

\vskip 0.1cm
\smallskip

\noindent {\textbf{Keywords}:}  Cayley graph; vertex-transitive graph;  $2$-extendablility; matching; edge-connectivity.

\vskip 0.1cm
\smallskip

\noindent \textbf{MSC 2010:} 05C70.
\end{abstract}

\section{Introduction}

Throughout this paper, all graphs are assumed to be connected and of even order. The order of a graph $G$ is written as $|G|$ and its size is denoted by $||G||$.  In mathematics, a Cayley graph is a graph that encodes the abstract
structure of a group. Let $\Gamma$ be a group and $S$ be an
inverse-closed generating set of $\Gamma$. The \emph{Cayley graph}
$G=G(\Gamma, S)$ on $\Gamma$ with respect to the connecting set $S$  is constructed as follows. Its vertex-set
$V(G)=\Gamma$ and for any $x, y\in \Gamma$, $x$ is adjacent to $y$
in $G$ if and only if $xy^{-1}\in S$. A graph $G$ is called {\em vertex-transitive} if for any two vertices $x$, $y$ in $V(G)$, there exists an automorphism $\psi$ of $G$ such that
$\psi(x) = y$. It is known that Cayley graphs are vertex-transitive (\cite[Proposition 16.2]{Biggs}). Because of this and their easy
construction, Cayley graphs are widely used in the design of
networks. As we see, the topological structures of many famous
networks are Cayley graphs \cite{Annexstein90}. On the other hand,
many authors have constructed lots of Cayley graphs with good
properties via specific groups, see \cite{Brunat99, Carlsson85}. In \cite{Konstantinova08}, this survey paper
presents the historical development of some problems on Cayley
graphs which are interesting to graph and group theorists such as
Hamiltonicity or diameter problems, to computer scientists and
molecular biologists such as pancake problem or sorting by
reversals, to coding theorists such as the vertex reconstruction
problem related to error-correcting codes but not related to Ulam's
problem.

A graph with at least $2n+2$ vertices is said to be
\emph{n-extendable} if it has a matching of size $n$ and every such
matching can be extended to a perfect matching.   The next lemma (Theorem 5.5.24 in \cite{LovaszandPlummer86})
essentially implies that a vertex-transitive graph is 1-extendable, in which, the definitions of elementary bipartiteness and bicriticality  can be found in \cite{LovaszandPlummer86}.

\begin{Lemma}[\cite{LovaszandPlummer86}] \label{1-extendable}
Let $G$ be a vertex-transitive graph. Then
it is either elementary bipartite or bicritical.
\end{Lemma}
 Recall that a Cayley graph is vertex-transitive. The following result arises immediately. 
\begin{cor}
Every Cayley graph is 1-extendable.
\end{cor}

An $n$-extendable graph is $(n-1)$-extendable, but the converse is
not true \cite{Plummer80}. Therefore, it is natural to consider more 
higher extendability of Cayley graphs. 
In \cite{Chan95}, the authors first classified the
2-extendable Cayley graphs on abelian groups, and posed the problem of
characterizing all 2-extendable Cayley graphs. 
Thereafter, the 2-extendability of Cayley graphs on specific groups,  
such as Dihedral groups \cite{Chen92},
Dicylic groups \cite{Bai13}, generalized dihedral groups
\cite{Mik09}, Quasi-abelian groups \cite{Xinggao} and etc, has been investigated.  The results can be summarized as follows, in which, $Z_{4n}(1,4n,2n)$
stands for the Cayley graph on $Z_{4n}$, the additive group modulo
$4n$, with respect to the connecting set $S=\{1,4n-1,2n\}$. $Z_{2n}(1,2n-1)$,
$Z_{2n}(1, 2, 2n-1, 2n-2)$, $Z_{4n+2}(2,4n,2n+1)$ and
$Z_{4n+2}(1,4n+1,2n,2n+2)$ are defined similarly.

\begin{thm}[\cite{Chan95, Chen92, Mik09}]\label{original}
Let $G$ be a Cayley graph on an abelian group or a Dihedral group or
a generalized dihedral group. Then $G$ is 2-extendable if and only if
it is not isomorphic to any of the following circulant graphs:

\smallskip

 {\rm(i)} $Z_{2n}(1,2n-1), n\geq3$;
 
 \smallskip

 {\rm(ii)} $Z_{2n}(1, 2, 2n-1, 2n-2), n\geq3$;
 
 \smallskip

 {\rm(iii)} $Z_{4n}(1,4n-1,2n), n\geq2$;
 
 \smallskip

 {\rm(iv)} $Z_{4n+2}(2,4n,2n+1), n\geq1$;
 
 \smallskip

 {\rm(v)} $Z_{4n+2}(1,4n+1,2n,2n+2), n\geq1$.
\end{thm}

\begin{thm}[\cite{Bai13}]
Let $G$ be a Cayley graph on a Dicylic group. Then $G$ is
2-extendable.
\end{thm}

From the above results, we can see that, despite of $Z_{2n}(1,2n-1)$ (which is indeed a cycle),  the other exceptional cases are
non-bipartite.
Motivated by this, we first consider 2-extendability of bipartite Cayley graphs and prove a more stronger result for vertex-transitive graphs.

\begin{thm}\label{bipartite}
A bipartite vertex-transitive graph is 2-extendable if and only if  it is not a cycle.
\end{thm} 

As a Cayley graph is vertex-transitive, we have the following result.
\begin{thm}
A bipartite Cayley graph is 2-extendable if and only if  it is not a cycle.
\end{thm}

For the 2-extendability of non-bipartite vertex-transitive graphs, Wuyang Sun and Heping Zhang showed the following result.

\begin{thm}[\cite{sun}]\label{sun}
 A non-bipartite vertex-transitive with degree $k\geq 5$ is
2-extendable.
\end{thm}

Obviously, the above result holds for Cayley graphs. Hence one only needs to classify 2-extendable non-bipartite Cayley graphs of minimum degrees three and four. We are going to solve the case of minimum degree three. Indeed, we also prove a more stronger result for vertex-transitive graphs.

\begin{thm}\label{maintheorem}
A cubic non-bipartite  vertex-transitive graph with girth $g$ is
2-extendable if and only if $g\geq 4$ and it does not isomorphic to  $Z_{4n}(1,4n-1,2n)$ or
$Z_{4n+2}(2,4n,2n+1)$ with $n\geq 2$ or the Petersen graph.
\end{thm}

As the Petersen graph is not a Cayley graph (\cite{Godsil}, Lemma 3.1.3), we have the following.

\begin{thm}
A connected cubic non-bipartite  Cayley graph with girth $g$ is
2-extendable if and only if $g\geq 4$ and it does not isomorphic to  $Z_{4n}(1,4n-1,2n)$ or
$Z_{4n+2}(2,4n,2n+1)$ with $n\geq 2$.
\end{thm}

Different from before, our proofs do not rely on the structures of the specific groups, but several kinds of connectivities of vertex-transitive graphs, such as edge-connectivity, restricted edge-connectivity, cyclic edge-connectivity and uniform cyclically edge-connectivity, will be used and play important  roles. Results related to these will be presented in Section 2. In Section 3, by using the generalized Hall's theorem and the strengthened Tutte's theorem,  we finally obtain the classifications of 2-extendable bipartite and cubic non-bipartite vertex-transitive graphs respectively. It is worth to note that the 2-extendabilities of toroidal fullerenes, generalized Petersen graphs and cyclically 5-edge-connected planar graphs are also needed.

\section{Some results related to several kinds of edge-connectivities of vertex-transitive graphs}

In this section, we present some results related to several kinds of edge-connectivities, including edge-connectivities, cyclic edge-connectivities,  restricted edge-connectivities and uniform edge-connectivities, of vertex-transitive graphs. As we will see, they play important roles in the classification.

\subsection{Edge-connectivity and super-edge-connectivity}

An edge set $S\subseteq E(G)$ is called an \emph{edge-cut} if there
exists
 $X\subseteq V(G)$ such that $S$ is the set of edges between $X$ and $\overline{X}$, where $\overline{X}:=V(G)\backslash X$. The {\em edge-connectivity} $\lambda(G)$ of $G$
is the minimum cardinality over all edge-cuts of it. It is easy to see that $\lambda(G)\leq \delta(G)$, where $\delta(G)$ denotes the minimum degree of $G$. Mader proved the following result, which essentially says that a connected vertex-transitive graph is maximally edge-connected. 

\begin{thm}[\cite{mader}]\label{mader}
If $G$ is a $k$-regular connected vertex-transitive graph, then $\lambda(G)=k$.
\end{thm}
 The following theorem goes a step further by characterizing the minimum edge-cuts of vertex-transitive graphs, where a {\em clique} is a subset of vertices such that every two distinct vertices in it are adjacent.

\begin{thm}[Lemma 5.5.26, \cite{LovaszandPlummer86}]\label{edgecut}
Let $G$ be a $k$-regular connected vertex-transitive graph. Then $G$
is $k$-edge-connected and either

\smallskip

 {\rm(i)} every minimum edge-cut of $G$ is the star of a point,
or

\smallskip

 {\rm(ii)} $G$ arises from a (not necessarily simple) vertex-
and edge-transitive $k$-regular graph $G_{0}$ by a $k$-clique
insertion at each point of $G_{0}$. Moreover, every minimum edge-cut
of $G$ is the star of a vertex of $G$ or a minimum edge-cut of
$G_{0}$.
\end{thm}

We call an edge-cut \emph{trivial} if it isolates a vertex and
\emph{non-trivial} otherwise.  The following corollary arises
immediately.

\begin{cor}\label{notriangletrivial}
Let $G$ be a $k$-regular connected vertex-transitive graph without
$k$-cliques. Then $G$ is $k$-edge-connected and every $k$-edge-cut of it is trivial.
\end{cor}

 For a graph $G$, if every minimum edge-cut of it is trivial, then we say it is {\em super-edge-connected} (or simply
{\em super-$\lambda$}). J. Meng has presented a characterization of a vertex-transitive graph to be super-$\lambda$ with respect to the cliques. As we know, a bipartite graph of minimum degree at least three is neither a complete graph nor a cycle and further does not contain $k$-cliques ($k\geq 3$). Hence the following result will be used  in proving the 2-extendability of bipartite vertex-transitive graphs.

\begin{thm}[\cite{meng}]\label{superlambda}
Let $G$ be a $k$-regular connected vertex-transitive graph which is neither a complete graph nor a cycle. Then
$G$ is super-$\lambda$ if and only if it does not contain $k$-cliques.

\end{thm}

\subsection{Restricted (super restricted) edge-connectivity}


 An edge set $F\subseteq E(G)$ is called a \emph{restricted edge-cut} if $G-F$ is disconnected and contains no isolated vertices. We define the \emph{restricted edge-connectivity}, denoted by $\lambda^{(2)}(G)$, to be the minimum cardinality of all
restricted edge-cuts. For $e=uv\in E(G)$, let $\xi_{G}(e)=d(u)+d(v)-2$ be the \emph{edge-degree} of $e$, and let $\xi_{2}(G)$=min$\{\xi_{G}(e): e\in E(G)\}$ be the \emph{minimum edge-degree} of $G$.
A graph $G$ is \emph{optimal-$\lambda^{(2)}$} if
$\lambda^{(2)}(G)=\xi_{2}(G)$.  In \cite{Hellwig}, the authors presented some  connections between optimal-$\lambda^{(2)}$, super-edge-connected, and maximally edge-connected graphs. For connected vertex-transitive graphs, Xu has studied behavior of the parameter $\lambda^{(2)}(G)$ and obtains the following result.

\begin{thm}[\cite{Xujunming2000}]\label{Xu2000}
Let $G$ be a connected $k$-regular vertex-transitive graph of order at least four.
 If its order is odd or
it does not contain triangles, then it is optimal-$\lambda^{(2)}$, that is, $\lambda^{(2)}(G)=2k-2$.
\end{thm}

Further, an optimal-$\lambda^{(2)}$ graph is
called \emph{super restricted edge-connected} (or in short
\emph{super-$\lambda^{(2)}$}) if every minimum restricted edge-cut isolates an edge. The following result on the super restricted edge-connectivity will be used in proving the super cyclically edge-connectivity of cubic vertex-transitive graphs of girth five.

\begin{thm}[\cite{Wang04}]\label{Wang04}
If G is a connected vertex-transitive graph with degree $k > 2$ and
girth $g > 4$, then it is super-$\lambda^{(2)}$.
\end{thm}

\subsection{Cyclic (super cyclically) edge-connectivity}

 For $\emptyset\neq X \subset V(G)$,
we denote  $\partial(X)$  by the set of edges of $G$ with one end in
$X$ and the other end in $\overline{X}$ and call it the
\emph{edge-cut associated with $X$}. Let $d(X)=|\partial(X)|$ and 
$\zeta(G)=min\{d(X)|X\subseteq V(G) \text{ and } X \text{ induces a
shortest cycle in } G$\}. For simplicity, we also use $\partial(G')$
and $d(G')$ to substitute for $\partial(V(G'))$ and $d(V(G'))$,
respectively, for a subgraph $G'$ of $G$. A \emph{cyclic edge-cut}
of a graph $G$ is a subset of $E(G)$, the removal of which separates two
cycles. If $G$ has a cyclic edge-cut, then it is called
\emph{cyclically separable}. For a cyclically separable graph $G$,
the \emph{cyclic edge-connectivity $c\lambda(G)$} is the
cardinality of a minimum cyclic edge-cut of $G$. Wang and Zhang
\cite{wang} have shown that $c\lambda(G)\leq \zeta(G)$ for any graph with
a cyclic edge-cut. If $c\lambda(G)= \zeta(G)$, then $G$ is called
\emph{cyclically optimal}. The next result shows that a cubic vertex-transitive graph is cyclically optimal.

\begin{thm}[\cite{Nedela95}]\label{Nedela95}
Let $G$ be a cubic vertex-transitive or edge-transitive graph with
girth $g$. Then $c\lambda(G)=g$.
\end{thm}

 If it happens that the removal of any
minimum cyclic edge-cut of a graph  results in a component which is a shortest cycle, then we call the graph
\emph{super cyclically edge-connected}. For the super cyclically edge-connectivity of cubic vertex-transitive graphs, the authors in \cite{Zhangzhao2011} have proved that a connected cubic vertex-transitive graph with
$g(G)\geq7$ is super cyclically edge-connected.
They also showed that the condition
$g(G)\geq7$ is necessary by exhibiting a vertex-transitive graph
with girth six which is  not super cyclically edge-connected. We are going to show that a cubic vertex-transitive graph of
girth five is also super cyclically edge-connected. Before proving this, several results and notations are needed.

\begin{Lemma}\label{order10and12}
Let $G$ be a cubic vertex-transitive graph of girth five. If in addition, $|G|=10$ or 12, then it is super cyclically 5-edge-connected.
\end{Lemma}
\begin{proof}
By Theorem \ref{Nedela95}, $c\lambda(G)=5$.  Assume $F$ is a cyclic 5-edge-cut of $G$. Then $G-F$ has exactly two components, denoted by $G_{1}$ and $G_{2}$, with each containing cycles. Since $d(G_{1})=3|G_{1}|- 2 ||G_{1}||=5$, we have $|G_{1}|$ is odd. On the other hand, since $G$ is of girth five, each component has at least five vertices. Hence if $|G|=10$, then $G-F$ has exactly two components and each is a cycle of length five. If $|G|=12$, then one of $G_{1}$ and $G_{2}$ contains exactly five vertices and further a cycle of length five, we are done.
\end{proof} 
The next lemma tells that the real-valued set function $d$ on the subsets of the vertex set, which is defined at the first paragraph in this subsection, is submodular.

\begin{Lemma}\label{submodular}
For any two vertex subsets $X$ and $Y$ in $V(G)$, 
$$d(X\cup Y)+d(X\cap Y)\leq d(X)+d(Y).$$
\end{Lemma}

\begin{proof}
It is straightforward to verify that every edge counted in $d(X\cup Y)+d(X\cap Y)$ is also counted in $d(X)+d(Y)$ and every edge counted in both $d(X\cup Y)$ and $d(X\cap Y)$ is also counted in both $d(X)$ and $d(Y)$.
\end{proof}

An \emph{imprimitive block} of $G$ is a proper non-empty subset $X$
of $V(G)$ such that for any automorphism $\psi$ of $G$, either
$\psi(X)=X$ or $\psi(X)\cap X=\emptyset$.

\begin{thm}[\cite{Tindell96}]\label{imprimitive}
Let $G$ be a vertex-transitive graph and $H$ be the subgraph of $G$
induced by an imprimitive block of $G$. Then $H$ is
vertex-transitive.
\end{thm}

A vertex subset $X$ is called  a \emph{cyclic edge-fragment}, if $\partial{(X)}$ is a minimum cyclic edge-cut. Let $X$ be a cyclic edge-fragment. If neither $X$ nor $\overline{X}$ induces a shortest cycle,
then we call $X$ a \emph{super cyclic edge-fragment}.  A super cyclic edge-fragment with the
minimum cardinality is called a \emph{super cyclic edge-atom}. By definition, if $X$ is a super atom, then $|\overline{X}|\geq |X|$. For simplicity of statement,
we shall use super atom to stand for super  cyclic edge-atom in this paper. 
For any two disjoint vertex subsets $X$ and $Y$ in $V(G)$, let $E(X,Y)$ denote the set of edges between $X$ and $Y$. 


\begin{Lemma}\label{super cyclically 5-edge-connected}
Let $G$ be a cubic vertex-transitive graph of girth five. Then $G$ is super cyclically 5-edge-connected.
\end{Lemma}

\begin{proof}

As we will see, that $G$ does not contain $H$ (see Figure \ref{351} (left)) as a subgraph is pivotal to prove the super cyclically 5-edge-connectivity. Hence we will first consider the properties of $G$ when it contains $H$ as a subgraph and obtain the following claim.

\smallskip

\vskip0.1cm

\noindent {\bf Claim.} If $G$ contains $H$ as a subgraph, then $|G|=10$ or 12.
 
\vskip0.1cm

\smallskip

Since $G$ is of girth five, $\{d,g,e,z\}$ induces a matching of size two in $G$. In other words, $a$ has two neighbors $b$ and $c$ satisfying that  $(N(b)\cup N(c) )\setminus \{a\}$ induces a matching of size two, where $N(v)$ denotes the neighborhood of a vertex $v$ in $G$. By the vertex-transitivity, every vertex of $G$ has such a property. 

\begin{figure}[!htbp]
\begin{center}
\includegraphics[totalheight=3.2cm]{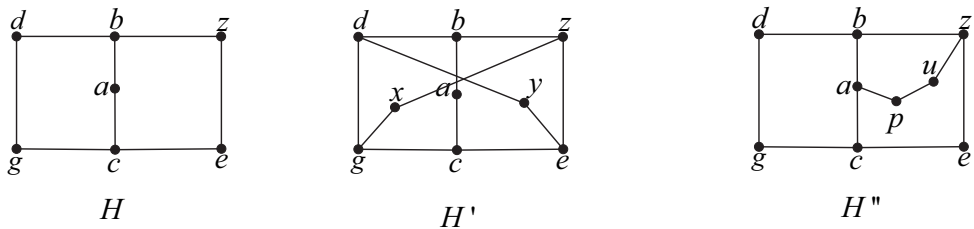}
 \caption{\label{351}The illustration for the proof.}
\end{center}
\end{figure}

We now consider the vertex $b$.  If $d$ and $z$ has the property that $(N(d)\cup N(z)) \setminus \{b\}$ induces a matching of size two, then also since $G$ is of girth five, $y$ should be adjacent to $e$ and $x$ must be adjacent to $g$, where $y$ and $x$ are neighbors of $d$ and $z$ respectively. Denote the appearing subgraph of $G$ until now by $H'$ (Figure \ref{351} (middle)). Since $|H'|=9$ is odd and $G$ is of even order, $V(G)\setminus V(H')\neq \emptyset$. Therefore, $\partial(H')$ is an edge-cut. By Theorem \ref{mader}, $d(H')\geq 3$. On the other hand, since $H'$ has  three vertices ($x$, $y$ and $a$) of degree two, $G[V(H')]$ has at most three vertices of degree two, where $G[S]$ denotes the subgraph of $G$ induced by $S\subseteq V(G)$. It follows that $d(H')\leq 3$. Hence $d(H')=3$ and $\partial(H')$ is an edge-cut of size three. By Corollary \ref{notriangletrivial}, $\partial(H')$ isolates a vertex. Therefore, $|G|=10$.  If $d$ and $a$ or $z$ and $a$ has the property that $(N(d)\cup N(a)) \setminus \{b\}$ or  $(N(z)\cup N(a)) \setminus \{b\}$ induces a matching of size two, by symmetry, we may suppose that $z$ and $a$ has such a property, then by a similar argument as above, we obtain a subgraph $H''$ of $G$ (Figure \ref{351} (right)). 

 \begin{figure}[!htbp]
\begin{center}
\includegraphics[totalheight=2.7cm]{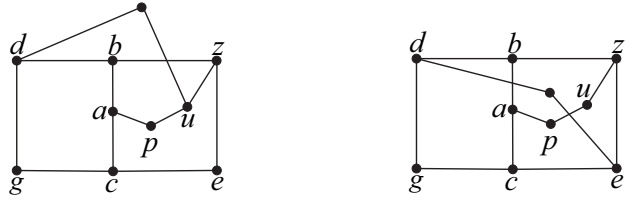}
 \caption{\label{352} $F'$ (left) and $F''$ (right).}
\end{center}
\end{figure}

 If $H''$ is not an induced subgraph, then there are at most three vertices of degree two in $G[V(H'')]$, by a similar argument as for $H'$, we can show that $|G|=10$. If $H''$ is an induced subgraph, then we consider the vertex $z$. If $b$ and $u$ (resp. $b$ and $e$) has the property that $(N(b)\cup N(u)) \setminus \{z\}$ (resp. $(N(b)\cup N(e)) \setminus \{z\}$) induces a matching of size two, then we obtain a subgraph $F'$ (resp. $F''$) of $G$ shown in Figure \ref{352} (left) (resp. Figure \ref{352} (right)). Let $F$ denote $F'$ or $F''$. We can see that there are four vertices of degree two in $F$. If $F$ is not an induced subgraph of $G$, then $d(F)\leq 2$. Since $G$ is 3-edge-connected, $V(G)\setminus V(F)=\emptyset$. Consequently, $|G|=10$. If $F$  is an induced subgraph of $G$, then $\partial{(F)}$ is an edge-cut of size four. By Theorem \ref{Wang04}, $\partial{(F)}$   isolates an edge. It follows that $|G|=12$. We are left to the case that $u$ and $e$ has the property that $(N(u)\cup N(e)) \setminus \{z\}$ induces a matching of size two.  If so, then $u$ must be adjacent to a vertex in $N(c)\setminus \{e\}$, that is, $u$ is adjacent to $a$ or $g$, contradicting that $H''$ is an induced subgraph of $G$. This finally completes the proof of the claim.

\vskip 0.1cm

\smallskip

Combining the above claim and Lemma \ref{order10and12}, we obtain that if $G$ contains $H$ as a subgraph, then it is super cyclically 5-edge-connected. Hence in the following, we assume that $G$ does not contain $H$ as a subgraph. By Theorem \ref{Nedela95}, $G$ is cyclic 5-edge-connected. Suppose by the contrary that $G$ is not super cyclically
5-edge-connected. Then $G$ contains a super atom. Let  $X$ be any given super atom. By definition, $|X|\geq 6$. We  claim that there are only degree-2 and degree-3 vertices in $G[X]$. Otherwise, suppose $x$ is a vertex of degree one in $G[X]$ and $y$ be the neighbor of $x$ in $G[X]$. Then $\partial(X\setminus\{x\})$ is a cyclic edge-cut of size four, contradicting that $G$ is cyclic 5-edge-connected.
Further, we show that the set of degree-2 vertices in $G[X]$ is an  independent set.
Suppose by the contrary that there are two adjacent degree-2
vertices denoted by $u$ and $v$ in $G[X]$. Let $X'=X \backslash
\{u,v\}$. Then $|X'|\geq4$ and $d(X')$=5. If $G[X']$ contains cycles, then $\partial(X')$ is a cyclic 5-edge-cut. Consequently, it must be a cycle of length five since $X$ is a super atom. It follows that we obtain $H$ as a subgraph, a contradiction. If $X'$ does not contain a cycle, then $G[X']$ contains at most $|X'|-1$ edges. Moreover, 
$$5=d(X')\geq 3|X'|-2(|X'|-1)=|X'|+2\geq6,$$
 a contradiction.

Now we show that $|X|\geq 9$.  In $G[X]$, on the one hand,  the five  independent degree-2
vertices send out ten edges to the other vertices; on the other hand, the other vertices are of degree three. Hence there are at least four vertices of degree three in $G[X]$, which implies that $|X|\geq 9$.

Finally, we prove that every super atom is an imprimitive block. If this holds, then by Theorem \ref{imprimitive},  $G[X]$ is vertex-transitive. But there are degree-2 and degree-3 vertices in it, a contradiction.

To prove this, it suffices to show that for two super atoms $X$ and $Y$, either $X=Y$
or $X\cap Y=\emptyset$. Suppose not. That is, $X\neq Y$ and $X\cap Y\neq \emptyset$. Let $A=X\cap Y\neq \emptyset$, $B=X\backslash
Y$, $C=Y\backslash X$ and $D=V(G)\backslash (X \cup Y)$. Since $Y$
is a super atom, $\partial(Y)=\partial(\overline{Y})=\partial{(B\cup D)}$ is a cyclic 5-edge-cut. Also since $X=B\cup A$ is a super atom, $|D|\geq |A|$
holds. By symmetry, we may assume that
$|E(A,C)|\leq |E(A,B)|$. Then 
\begin{equation*}
\begin{aligned}
d(C)&=|E(C, A)|+|E(C, B)|+|E(C, D)|\\
&=|E(A, C)|+|E(C, B)|+|E(C, D)|\\&\leq|E(A, B)|+|E(C, B)|+|E(C, D)|+|E(A, D)|\\&=d(Y)= c\lambda(G)=5.
\end{aligned}
\end{equation*}

We claim that $|A|\geq 4$. Suppose by the contrary that $|A|\leq 3$.
Then $|C|\geq 6$ by $Y=A\cup C$ and $|Y|\geq 9$ ($Y$ is a super atom). Hence
$$||G[C]||=\frac{3|C|-d(C)}{2}\geq
\frac{3|C|-5}{2}\geq |C|.$$
 It follows that $G[C]$ contains cycles. This contradicts
that $Y$ is a super atom, $|C|\geq 6$ and $A\neq \emptyset$.

Since $|D|\geq |A|\geq 4$,  if $G[D]$ (resp. $G[A]$) does not contain cycles, we have 
$$d(D)\geq 3|D|-2(|D|-1)=|D|+2\geq6$$
$$({\rm resp.} ~~d(A)\geq 3|A|-2(|A|-1)=|A|+2\geq6).$$
If $G[D]$ contains cycles, then $\partial(D)$ is a cyclic edge-cut and $d(D)\geq 5$ follows. In a word,  $d(D)\geq 5$.
By Lemma \ref{submodular}, we have 
$$d(X\cup Y)+d(X\cap Y)\leq d(X)+d(Y);$$
that is,
$$d(D)+d(A)\leq d(X)+d(Y)=10.$$
Therefore,  $d(A)\leq 5$ by $d(D)\geq 5$.
Further, by the above argument, we can see that $G[A]$
contains cycles. Hence $\partial(A)$ is a cyclic edge-cut and $|A|\geq5$ follows. Since $Y$ is a super atom, $|A|\leq5$. Hence $|A|=5$ and $G[A]$ is isomorphic to a cycle of length five. It follows that  $|C|\geq
4$. By 
$$\partial(C)=E(C,A)\cup E(C,B)\cup E(C,D),$$
we have,

$$d(C)=|E(C,A)|+|E(C,B)|+|E(C,D)|.$$
Therefore, by $|E(C,A)|\leq |E(A,B)|$, we have

\begin{equation*}
\begin{aligned}
d(C)&=|E(C,A)|+|E(C,B)|+|E(C,D)|\\&\leq|E(A,B)|+|E(C,B)|+|E(C,D)|+|E(A,D)|\\&=d(Y)=5.
\end{aligned}
\end{equation*}

Since $|C|\geq 4$ and $d(C)\leq 5$, by a similar argument as for $G[A]$, we obtain that $G[C]$ contains cycles and further $G[C]$ must be a cycle of length five, too. Consequently,
$|Y|=10$ is even. Then  $d(Y)=3|Y|-2||G[Y]||$ is even, but $d(Y)=5$ is odd, a contradiction. 
\end{proof}

 \begin{figure}[!htbp]
\begin{center}
\includegraphics[totalheight=6.5cm]{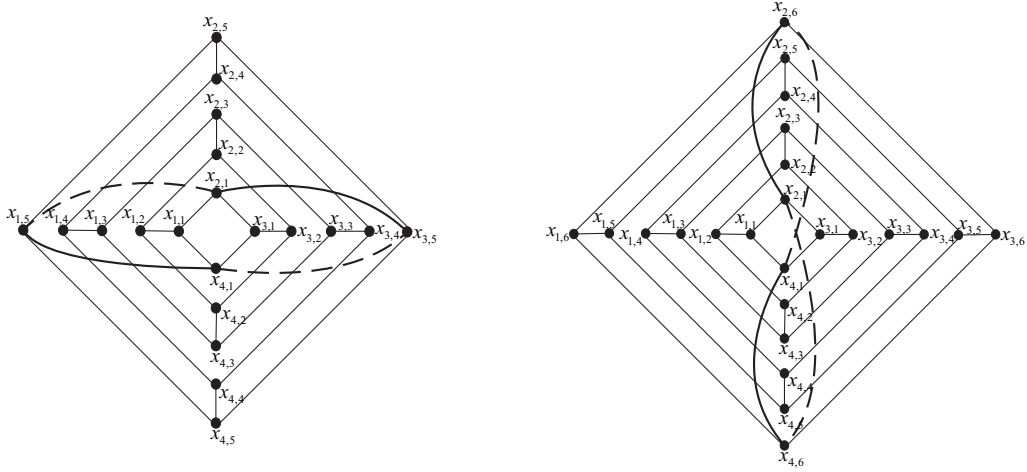}
 \caption{\label{g4notsuper} Graphs $T_{5}$ (left) and $T_{6}$ (right), where $M$ is chosen either the set of bold edges or the dashed edges.}
\end{center}
\end{figure}

For the super cyclically edge-connectivity of cubic vertex-transitive graphs of girth four, we obtain a similar but  little weaker result. Precisely, certain structure (adjacent quadrangles) is forbidden and a class of graphs ($T_{m}$,  $m\geq 2$)  which do not contain adjacent quadrangles are excluded. 
Two quadrangles in $G$ are \emph{adjacent} if they share common vertices
or edges. The class of graphs $T_{m}$ which do not contain adjacent quadrangles are  defined as follows, in which, the first  subscripts of $x$ are taken module $4$.
For odd $m=2k+1$,  $V(T_{m})=\{x_{i,j}|1\leq i \leq 4, 1\leq j\leq m\}$ and $E(T_{m})=\{x_{i,j}x_{i+1,j} |1\leq i \leq 4, 1\leq j \leq m\}\cup \{x_{1,2j-1}x_{1,2j} | 1\leq j\leq k\}\cup \{x_{3,2j-1}x_{3,2j} | 1\leq j\leq k\} \cup \{x_{2,2j}x_{1,2j+1} | 1\leq j\leq k\}\cup \{x_{4,2j}x_{4,2j+1} | 1\leq j\leq k\}\cup M$, where $M=\{x_{1,m}x_{2,1}, x_{3,m}x_{4,1}\}$ or $\{x_{1,m}x_{4,1}, x_{3,m}x_{2,1}\}$. For even $m=2k$,    $V(T_{m})=\{x_{i,j}|1\leq i \leq 4, 1\leq j\leq m\}$ and $E(T_{m})=\{x_{i,j}x_{i+1,j} |1\leq i \leq 4, 1\leq j \leq m\}\cup \{x_{1,2j-1}x_{1,2j} | 1\leq j\leq k\}\cup \{x_{3,2j-1}x_{3,2j} | 1\leq j\leq k\} \cup \{x_{2,2j}x_{1,2j+1} | 1\leq j\leq k\}\cup \{x_{4,2j}x_{4,2j+1} | 1\leq j\leq k\}\cup M$, where $M=\{x_{2,m}x_{2,1}, x_{4,m}x_{4,1}\}$ or $\{x_{2,m}x_{4,1}, x_{4,m}x_{2,1}\}$. Please see Figure \ref{g4notsuper} for example.

\begin{Lemma}\label{g4superaaa}
Let $G$ be a cubic vertex-transitive graph of girth four and do not contain adjacent quadrangles. If in addition, $G$ is not isomorphic to $T_{m}$ for some integer $m$, then $G$ is super cyclically
4-edge-connected.
\end{Lemma}
\begin{proof}
First, we show that $G$ cannot contain the left graph in Figure \ref{g4super} as a subgraph. Suppose not. Since $G$ is cubic and does not contain adjacent quadrangles, every vertex is contained in exactly one quadrangle. Consequently, $w$ must lie in a new quadrangle other than the one containing $z$. We then obtain a subgraph of $G$, shown in the right one of Figure \ref{g4super}.  Let $Q$ be the quadrangle containing $z$. We can see that $a$ and $z$ lie on the opposite position of $Q$ and, $w$ and $b$, the neighbors of them in $V(G)\setminus V(Q)$, are adjacent. By the vertex-transitivity, $w$ should have the same property as $z$, that is, if we denote the quadrangle containing $w$ be $Q'$, then $y$ lies on the opposite site of $w$ in $Q'$ and further the neighbors of $w$ and $y$ in $V(G)\setminus V(Q')$ are adjacent. It follows that  $y$ should be adjacent to $u$ (resp. $v$), resulting two adjacent quadrangles $Q$ and $uaby$  (resp. $Q$ and $vaby$), a contradiction.

 \begin{figure}[!htbp]
\begin{center}
\includegraphics[totalheight=3cm]{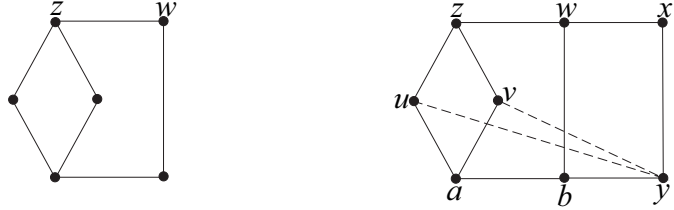}
 \caption{\label{g4super}Forbidden structures of $G$.}
\end{center}
\end{figure}

Next, we show that $G$ is super cyclically
4-edge-connected. Suppose by the contrary not. Then there is a super atom, denoted by $X$ with $|X|\geq 5$. By a completely similar examination as the proof for the case of girth five, we can show that the degree-2 vertices in $G[X]$ are independent and further $|X|\geq 7$. We also finish our proof by proving that every super atom is an imprimitive block. Suppose not. Then following the notations as in the case of girth five, we can show that $|A|\geq 3$ and further 
 both $G[A]$ and $G[C]$ are isomorphic to a cycle of length four.
Similarly as $G[C]$, we obtain that $G[B]$ is isomorphic to a cycle of length four, too. 
Since $X$ and $Y$ are super atoms, that is, $d(X)=d(Y)=4$, there are two edges connecting $A$ and $C$, and two edges connecting $A$ and $B$. Moreover, since the degree-2 vertices are independent in $G[X]$ and $G[Y]$, we obtain  a subgraph of $G$ as follows (see Figure \ref{g4substructure}), where the inner, middle and outer quadrangles stand for $G[B]$, $G[A]$ and $G[C]$ respectively.

 \begin{figure}[!htbp]
\begin{center}
\includegraphics[totalheight=4cm]{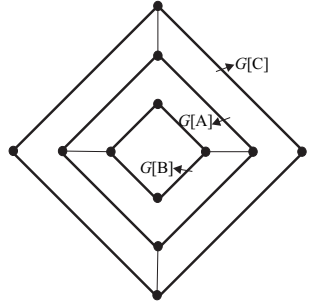}
 \caption{\label{g4substructure}$G[A\cup B \cup C]$.}
\end{center}
\end{figure}

Since $G$ is vertex-transitive and does not contain adjacent quadrangles,   $G[C]$ should have the same property as $G[A]$. That is, $G[C]$ is connected to two other different quadrangles like the way as $G[A]$. If $G[C]$ is connected to $G[B]$, then we obtain that $G$ is isomorphic to $T_{3}$; otherwise, $G[C]$ is connected to a new quadrangle, denoted by $Q'$. Continuously, since $Q'$ has the same property as $G[A]$, either we finish at $T_{4}$ or obtain another new quadrangle. As $G$ is finite, this process must end in a finite step. Therefore, we obtain $G$ is isomorphic to $T_{m}$ for some integer $m$, a contradiction.
\end{proof}

At the end of this subsection, we take a little time to show that $T_{m}$ is 2-extendable. The 2-extendability of toroidal fullerenes, whose definition can be found in \cite{Dong08}, will  be used. The results in that paper can be summarized as follows.

\begin{thm}[\cite{Dong08}]\label{Dong08}
A toroidal fullerene of girth at least four is 2-extendable.
\end{thm}

\begin{Lemma}\label{Tmextendable}
$T_{m}$ is 2-extendable.
\end{Lemma}

\begin{proof}

We make a bijection from  $V(T_{m})$ to the vertex set of a toroidal fullerene $H(2,m,0)$ or $H(2,m,1)$ according to the choice of $M$, see Figure \ref{toroidalfullerene} for $T_{5}$ as an example. Then we can easily check that $T_{m}$ is isomorphic to $H(2,m,0)$ or $H(2,m,1)$. By Theorem \ref{Dong08}, $T_{m}$ is 2-extendable.
 \begin{figure}[!htbp]
\begin{center}
\includegraphics[totalheight=5.5cm]{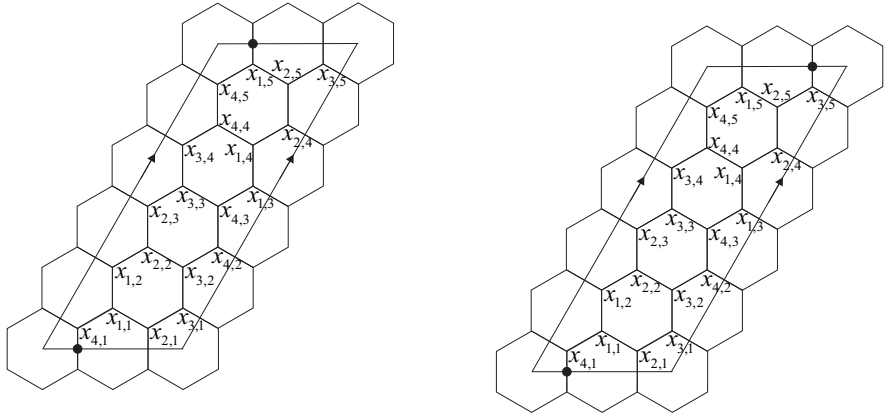}
 \caption{\label{toroidalfullerene} A bijection from $V(T_{5})$ to $H(2,m,0)$ (left) or $H(2,m,1)$ (right)  according to the choice of $M$.}
\end{center}
\end{figure}
\end{proof}

\subsection{Uniformly cyclically edge-connectivity} 

Another kind of cyclic edge-connectivity is also needed. We say that a cyclically $k$-edge-connected  cubic graph $G$ is
\emph{uniformly cyclically $k$-edge-connected}, denoted by $U(k)$, 
if and only if there are no removable edges in $G$. Note that an
edge in $G$ is removable if and only if it does not lie in a cyclic
$k$-edge-cut. Hence $G$ is $U(k)$ if and only if every edge lies in
a cyclic $k$-edge-cut. The authors in \cite{Aldred91} have characterized the uniformly cyclically 5-edge-connected cubic graphs with some restriction. In that paper, a class of graphs, named double ladder, are introduced.

\begin{defn}
An odd double ladder of length $k$ is a graph $G$ with vertex
set $V(G)=\{a_{i}: 1 \leq i \leq k\}\cup \{b_{i}: 1 \leq i \leq
2k+1\}\cup \{c_{i}: 1 \leq i \leq k+1\}$ and edge set
$E(G)=\{(a_{i}b_{2i}): 1 \leq i \leq k\}\cup \{(c_{i}b_{2i-1}): 1
\leq i \leq k+1\}\cup \{(a_{i}a_{i+1}): 1 \leq i \leq k-1\}\cup
\{(b_{i}b_{i+1}): 1 \leq i \leq 2k\}\cup \{(c_{i}c_{i+1}): 1 \leq i
\leq k\}\cup$ (an arbitrary matching between $\{a_{1}, b_{1},
c_{1}\}$ and $\{a_{k}, b_{2k+1}, c_{k+1}\}$).

An even double ladder of length $k$ is a graph $G$ with
vertex set $V(G)=\{a_{i}: 1 \leq i \leq k\}\cup \{b_{i}: 1 \leq i
\leq 2k\}\cup \{c_{i}: 1 \leq i \leq k\}$ and edge set
$E(G)=\{(a_{i}b_{2i}): 1 \leq i \leq k\}\cup \{(c_{i}b_{2i-1}): 1
\leq i \leq k\}\cup \{(a_{i}a_{i+1}): 1 \leq i \leq k-1\}\cup
\{(b_{i}b_{i+1}): 1 \leq i \leq 2k-1\}\cup \{(c_{i}c_{i+1}): 1 \leq
i \leq k-1\}\cup$ (an arbitrary matching between $\{a_{1}, b_{1},
c_{1}\}$ and $\{a_{k}, b_{2k}, c_{k}\}$).
\end{defn}

\begin{thm}[\cite{Aldred91}]\label{Aldred91}
Let $G$ be a $U(5)$ graph which contains a 5-cycle which doesn't
belong to a rosette. Then $G$ is either a double ladder or one of
$G_{1}$ and $G_{2}$ (see Figure \ref{rossete}).
\end{thm}

 \begin{figure}[!htbp]
\begin{center}
\includegraphics[totalheight=4.3cm]{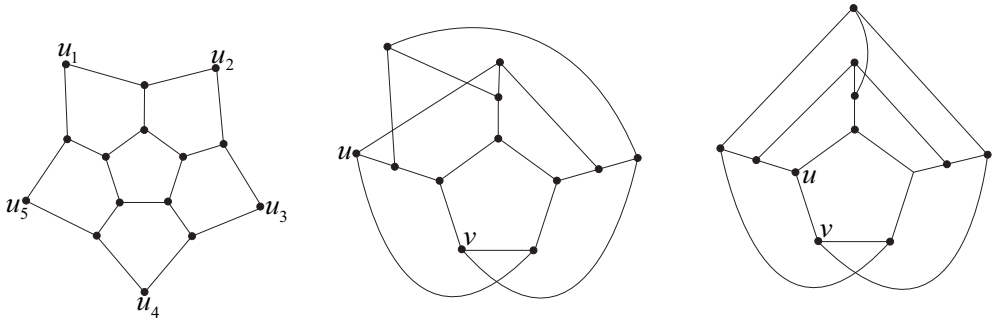}
 \caption{\label{rossete}The  rosette (left), $G_{1}$ (middle) and $G_{2}$ (right).}
\end{center}
\end{figure}

As we see, for a $U(5)$ graph, if a resette is excluded, then its structure is clear. Hence we first characterize a $U(5)$ and vertex-transitive graph containing a resette as a subgraph. 

\begin{Lemma}\label{dodecahedron}
Let $G$ be a vertex-transitive graph of girth 5 which is also in $U(5)$. If it contains a resette as a subgraph, then it is isomorphic to the dodecahedron.
\end{Lemma}
\begin{proof}
Let $H$ be a rosette in $G$ and $u_{1}, u_{2}, u_{3}, u_{4}$ and $u_{5}$ be the five vertices of degree two in $H$ (see Figure \ref{rossete} (left)).  Since $|H|$ is odd and $|G|$ is even, $\overline{V(H)}\neq \emptyset$. It follows that $\partial(H)$ is an edge cut. Moreover, $d(H)\geq 3$ by Theorem \ref{mader}. 

\vskip 0.1cm
\smallskip

\noindent  {\bf{Claim 1.}}  $H$ is an induced subgraph. 

\vskip 0.1cm
\smallskip

Suppose not. Then $d(H)\leq 3$. Furthermore, we have $d(H)=3$ by Theorem \ref{mader}. Moreover, by Corollary \ref{notriangletrivial}, $\overline{V(H)}$ contains a single vertex, denoted by $w$. $w$ is of degree three and it is adjacent to three of  $u_{1}, u_{2}, u_{3}, u_{4}$ and $u_{5}$. Therefore, it is adjacent to  $u_{i}$ and $u_{i+1}$ for some $i$, here and hereafter in the proof of this lemma, the subscripts are taken modulo 5. Then we obtain a quadrangle, a contradiction. This completes the proof of the claim.

\vskip 0.1cm
\smallskip

Let $v_{1}$, $v_{2}$, $v_{3}$, $v_{4}$ and $v_{5}$ be the neighbors of  $u_{1}, u_{2}, u_{3}, u_{4}$ and $u_{5}$ in $V(G)\setminus V(H)$ respectively.  

\vskip 0.1cm
\smallskip

\noindent  {\bf{Claim 2.}}  All $v_{i}'s$ are different from each other.

\vskip 0.1cm
\smallskip

Since $G$ is of girth five, $v_{i}\neq v_{i+1}$, $1\leq i \leq 5$. It follows that only two of them could be the same. Otherwise, we will find some $i$ ($1\leq i \leq 5$) with $v_{i}=v_{i+1}$. Suppose that  $v_{i}=v_{j}$ with $|i-j|\neq 1$. Then by symmetry, we may suppose that $v_{1}=v_{3}$.  Since $G$ is vertex-transitive, $G[N_{2}(x)]$ should be isomorphic to $G[N_{2}(y)]$, where $x$ and $y$ are shown in Figure \ref{rosettestructure} (left), and     $N_{2}(v)$ denotes the set of vertices of distance two with $v$ in $G$.  $G[N_{2}(x)]$ is a matching of size three, so is $G[N_{2}(y)]$. Hence  we obtain that $v_{3}$ is adjacent to $v_{5}$. Let $S=V(H)\cup \{v_{1}, v_{5}\}$. Then $d(S)=3$. By Theorem \ref{superlambda}, $\partial(S)$ isolates a singleton; that is, the three degree-2 vertices in $G[S]$ is connected to a common vertex $v_{2}=v_{4}$.  Now $G$ is known, we can check that $G[N_{2}(x)]$ is not isomorphic to $G[N_{2}(z)]$, please see Figure \ref{rosettestructure} (left), a contradiction.

\vskip 0.1cm
\smallskip

Since $d(H)=5$ and $|\overline{V(H)}|\geq 5$, we can easily check that $G[\overline{V(H)}]$ contains a cycle and further $\partial(H)$ is a cyclic 5-edge-cut. By Lemma \ref{super cyclically 5-edge-connected}, $G[\overline{V(H)}]$ is a cycle of length five. Hence we have $\overline{V(H)}=\{v_{1}, v_{2}, v_{3}, v_{4}, v_{5}\}$.
We will show that $v_{i}$ can only be adjacent to $v_{i+1}$ and $v_{i-1}$ in $G[\overline{V(H)}]$. Suppose not. By symmetry, we may suppose that $v_{2}v_{4}\in E(G)$. Then $v_{2}v_{1}\notin E(G)$ or $v_{2}v_{3}\notin E(G)$. If $v_{2}v_{1}\notin E(G)$, then $G[N_{2}(w)]$ is not isomorphic to $G[N_{2}(u)]$, a contradiction; If $v_{2}v_{3}\notin E(G)$, then $G[N_{2}(v)]$ is not isomorphic to $G[N_{2}(u)]$, a contradiction, too; Please see Figure Figure \ref{rosettestructure} (middle).

\begin{figure}[!htbp]
\begin{center}
\includegraphics[totalheight=4.5cm]{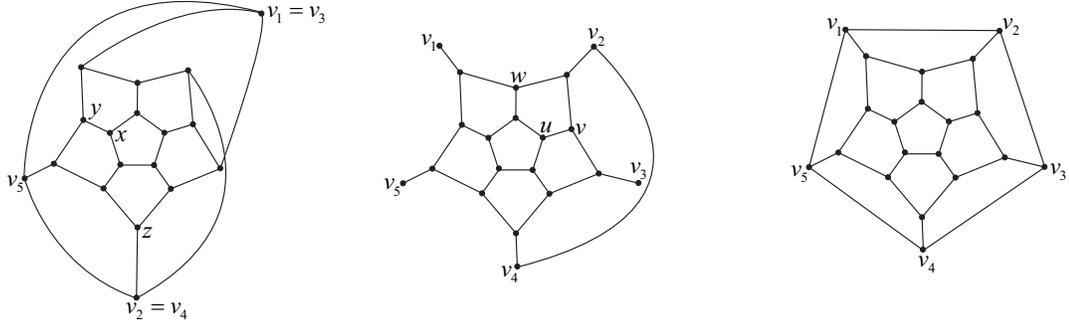}
 \caption{\label{rosettestructure}The  illustration for the proof.}
\end{center}
\end{figure}
By the above arguments, $v_{i}$'s are all different from each other and $v_{i}$ is adjacent to $v_{i+1}$ and $v_{i-1}$ in $G[\overline{V(H)}]$.  Thus we obtain that $G$ is isomorphic to the dodecahedron (see Figure \ref{rosettestructure} (right)), we are done.
\end{proof}

In \cite{Aldred91}, that when a double ladder is $U(5)$ is also characterized.

\begin{Lemma}[\cite{Aldred91}]\label{Aldred91L}
An odd double ladder of length $k$ is $U(5)$ if and only if $k\geq
2$ and the matching is $\{(a_{1}c_{k+1}), (b_{1}b_{2k+1}),
(c_{1}a_{k})\}$, or $k=3$ and the matching is $\{(a_{1}b_{7}),
(b_{1}c_{4}), (c_{1}a_{3})\}$ or $\{(a_{3}b_{1}),$ $(b_{7}c_{1}),
(c_{4}a_{1})\}$, or $k=4$ and the matching is $\{(a_{1}b_{7}),
(b_{1}a_{4}), (c_{1}c_{5})\}$. An even double ladder of length $k$
is $U(5)$ if and only if $k\geq 5$ and the matching is
$\{(a_{1}a_{k}), (b_{1}b_{2k}), (c_{1}c_{k})\}$.
\end{Lemma}

In the following, we will show the 2-extendability of a double ladder which is also $U(5)$ and vertex-transitive. The 2-extendabilities of generalized Petersen graph and a kind of planar graphs are needed.

\begin{defn}
The generalized Petersen graph $GP(n,k), n \geq 3, 1\leq k <
\frac{n}{2}$ is a cubic graph with vertex set $\{u_{i}, i \in
Z_{n}\} \cup \{v_{i}, i \in Z_{n}\}$, and edge set $\{u_{i}u_{i+1},
u_{i}v_{i}, v_{i}v_{i+k}; \\i\in Z_{n}\}$. 
\end{defn}

\begin{thm}[\cite{Schrag}]\label{Schrag}
$G(n,2)$ is 2-extendable if and only if $n\neq 5, 6, 8$.
\end{thm}

\begin{thm}[\cite{Lou91}]\label{Lou91}
If $G$ is a cubic, 3-connected, planar graph and, in addition, is
cyclically 5-edge-connected, then $G$ is 2-extendable.
\end{thm}

\begin{thm}\label{doubleladderextendable}
Let $G$ be a double ladder. If $G$ is $U(5)$ and vertex-transitive, then it is 2-extendable except for the Petersen graph.
\end{thm}

\begin{proof}
If an odd double ladder of
length $k$ is $U(5)$, then by Lemma \ref{Aldred91L}, it is an odd double ladder with $k\geq
2$ and the matching is $\{(a_{1}c_{k+1}), (b_{1}b_{2k+1}),
(c_{1}a_{k})\}$, or $k=3$ and the matching is $\{(a_{1}b_{7}),
(b_{1}c_{4}), (c_{1}a_{3})\}$ or $\{(a_{3}b_{1}), (b_{7}c_{1}),
(c_{4}a_{1})\}$, or $k=4$ and the matching is $\{(a_{1}b_{7}),
(b_{1}a_{4}), (c_{1}c_{5})\}$.  For an odd double ladder of length 3  and the matching is $\{(a_{1}b_{7}), (b_{1}c_{4}), (c_{1}a_{3})\}$ or $\{(a_{3}b_{1}), (b_{7}c_{1}),
(c_{4}a_{1})\}$, or of length 4 and the matching is $\{(a_{1}b_{7}), (b_{1}a_{4}), (c_{1}c_{5})\}$, we can easily check that $G[N_{2}(a_{1})]$ is not isomorphic to $G[N_{2}(a_{2})]$; that is, they are not vertex-transitive graphs. Hence  an odd double ladder which is also $U(5)$ and vertex-transitive  is an odd double ladder with $k\geq
2$ and the matching is $\{(a_{1}c_{k+1}), (b_{1}b_{2k+1}), (c_{1}a_{k})\}$.  If we make a mapping from $c_{i}$ to $u_{2i-1}$ with $1\leq i \leq k+1$ and from $a_{i}$ to $u_{2i}$ with $1\leq i\leq k$, then we can see that an odd double ladder with $k\geq
2$ and the matching is $\{(a_{1}c_{k+1}), (b_{1}b_{2k+1}),
(c_{1}a_{k})\}$ is isomorphic to the generalized Petersen graph $G(2k+1, 2)$. By
Theorem \ref{Schrag},  $G(2k+1, 2)$ is 2-extendable except $G(5, 2)$ which is the Petersen graph.

 If an even double ladder of length $k$ is $U(5)$, then we can draw it on the plane, please see Figure \ref{planar} for $k=7$ as an example.  Therefore, it is a planar graph. By Theorem \ref{Lou91}, it is 2-extendable.
 \begin{figure}[!htbp]
\begin{center}
\includegraphics[totalheight=5cm]{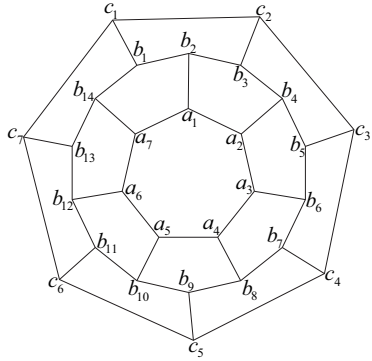}
 \caption{\label{planar} A drawing of an even ladder of length 7 in $U(5)$ on the plane.}
\end{center} 
\end{figure}
\end{proof}

\section{The 2-extendability of vertex-transitive graphs}
In this section, we first show the 2-extendability of bipartite vertex-transitive graphs. The generalized Hall's Theorem and a necessary condition of a graph to be $n$-extendable are used. As we see, if we substitute $k=0$ in the following theorem, then Hall's Theorem is obtained.

\begin{Lemma}[\cite{Plummer86}]\label{equivalent}
 Let $G$ be a connected graph with bipartition $(U,W)$ and
$k$ a positive integer such that $k\leq \frac{|V(G)|-2}{2}$. Then $G$ is $k$-extendable if and only if $|U|=|W|$ and for each non-empty subset $X$ of $U$ with $|X|\leq |U|-k$, $|N_G(X)|\geq |X|+k$.
\end{Lemma}

\begin{Lemma}[\cite{Plummer80}]\label{n+1connected}
An $n$-extendable graph is $(n+1)$-connected.
\end{Lemma}
\vspace{.3cm}

\noindent {\em Proof of Theorem \ref{bipartite}.} Let $G$ be a bipartite graph with bipartition $(X,Y)$. Since $G$ is
$k$-regular, $|X|=|Y|$.  

\vskip 0.1cm

\smallskip

Necessity. By Lemma \ref{n+1connected}, a 2-extendable graph is necessarily 3-connected and hence of minimum degree at least three. Consequently, it is not a cycle.

\vskip 0.1cm

\smallskip

Sufficiency. Suppose by the contrary that $G$ is not
2-extendable. Then by Theorem \ref{equivalent}, there exists $\emptyset \neq S\subset X$ with $|S|\leq |X|-2$
and $|N_{G}(S)|\leq |S|+1$. On the other hand,  $G$ is 1-extendable by Lemma \ref{1-extendable}. Furthermore, by Theorem \ref{equivalent}, $|N_{G}(S)|\geq |S|+1$. Therefore, $|N_{G}(S)|=|S|+1$. Denote $S'=S\cup
N_{G}(S)$. Then $d(S')=k$, $S'\neq \emptyset$ and $\overline{S'}\neq \emptyset$. Then by Theorem \ref{superlambda}, $\partial(S')$ isolates a singleton, which contradicts that $|S|\leq |X|-2$. This completes the proof. \hspace*{\fill}$\square$

\vspace{.3cm}

Next, we consider the 2-extendabilities of 
non-bipartite cubic vertex-transitive graphs. Another necessary condition  for a graph to be $n$-extendable is presented as follows.

\begin{Lemma}[\cite{Dean92}]{\label{n+t}}
Let $v$ be a vertex of degree $n+t$ in an n-extendable graph $G$.
Then $G[N(v)]$ does not contain a matching of size $t$.
\end{Lemma}

From it, we can obtain the following result.

\begin{cor}
Let $G$ be a cubic graph. If $G$ is 2-extendable, then it does not contain a triangle.
\end{cor}

A {\em fullerene} is a cubic plane graph with only pentagons and hexagons.  As we know, the dodecahedron is the smallest fullerene with only twelve pentagons. 

\begin{thm}[\cite{Zhangheping}]\label{Zhangheping}
Every fullerene graph is 2-extendable. In particular, the dodecahedron is 2-extendable.
\end{thm}

Two characterizations of vertex-transitive graphs by restricting some structures are also needed.
\begin{Lemma}[\cite{li2014}]\label{structure}
Let $G\neq K_{2}$ be a connected $k$-regular  non-bipartite
vertex-transitive graph. Let $S\subseteq V(G)$ with
$|S|=|\overline{S}|+2$, where $\overline{S}=V(G)\backslash S$. If in
addition, $\overline{S}$ is an independent set of $G$, then $G$ is
isomorphic to $Z_{4n+2}(1,4n+1,2n, 2n+2)$, $Z_{4n}(1,4n-1,2n)$ or
$Z_{4n+2}(2,4n,2n+1)$ or the Petersen graph.
\end{Lemma}

\begin{Lemma}[\cite{li2014}]\label{k3g4}
Let $G$ be a connected 3-regular  non-bipartite vertex-transitive
graph of girth 4. If $G$ has adjacent quadrangles, then it is
isomorphic to $Z_{4n}(1, 4n-1, 2n)$ or $Z_{4n+2}(2, 4n, 2n+1)$ with
$n\geq 2$.
\end{Lemma}

The Tutte's theorem, a stronger version of Tutte's theorem and a property of factor-critical graphs are used. We call a vertex set $S\subseteq V(G)$ \emph{matchable} to $G-S$ if
the (bipartite) graph $H_{s}$, which arises from $G$ by contracting
each
 component $c\in \mathcal{C}_{G-S}$ to a singleton and deleting all the edges inside $S$,
 contains a matching of $S$, where $\mathcal{C}_{G-S}$ denotes the set of the components of $G-S$. For $0\leq k \leq n$, a graph $G$ of order $n$ is said to be \emph{$k$-factor-critical}
 (\emph{$k$-fc} for short) if the removal of any $k$ vertices
 results in a graph with a perfect matching.
  
  \begin{thm}[\cite{Tutte47}] \label{tutte}
  A graph $G$ has a
perfect matching if and only if $c_{o}(G-U) \leq |U|$ for any $U \subseteq V(G)$, where
$c_{o}(G-U)$ is the number of odd components of $G-U$.
\end{thm}

\begin{thm} [Theorem 2.2.3, \cite{Diestel06}]\label{fc-component}
Every graph $G$ contains a set $S\subseteq V(G)$ with the following
two properties:

\smallskip

 {\rm (i)} $S$ is matchable to $G-S$;
 
 \smallskip

 {\rm (ii)} Every component of $G-S$ is 1-fc.
 
 \smallskip

\noindent Given any such set $S$, $G$ has a perfect matching if and
only if
  $|S| =|\mathcal{C}_{G-S}|$.
\end{thm}

\begin{Lemma}[\cite{O.Favaron}] \label{O.Favaron}
For $k\geq 1$, every $k$-fc graph of order $n>k$ is $k$-connected and
$(k+1)$-edge-connected.
\end{Lemma}

Now, it is ready to prove our main result.
\vspace{0.3cm}

\noindent{\em{Proof of Theorem \ref{maintheorem}.}} Necessity. If $g=3$, then by Lemma \ref{n+t}, it is not 2-extendable. By Theorems \ref{original} and \ref{Schrag}, $Z_{4n}(1,4n-1,2n)$, 
$Z_{4n+2}(2,4n,2n+1)$ with $n\geq 2$ and  the Petersen graph are not 2-extendable. We are done.

\vskip 0.1cm

\smallskip

Sufficiency.
Suppose by the contrary that $G$ is not 2-extendable. Then there exist two
independent edges $e_{1}$ and $e_{2}$ such that
$G'=G-V(e_{1})-V(e_{2})$ has no perfect matchings. By Theorem
\ref{fc-component}, there exists an $S'\subseteq V(G')$ satisfying
that $S'$ is matchable to $G'-S'$, $|\mathcal{C}_{G'-S'}|\geq |S'|+1$ and all components of $G'-S'$ are 1-fc. Since $|S'|$ and $|\mathcal{C}_{G'-S'}|$ have the same property, we have $|\mathcal{C}_{G'-S'}|\geq |S'|+2$. By Lemma \ref{1-extendable}, $G$ is 1-extendable. Consequently, $G''=G-V(e_{1})$ has a perfect matching. Hence $|\mathcal{C}_{G''-S''}|\leq |S''|= |S'|+2$ by Theorem \ref{tutte}, where $S''=S'\cup V(e_{2})$. We have 
$$ |S'|+2\leq |\mathcal{C}_{G'-S'}|=|\mathcal{C}_{G''-S''}|\leq  |S'|+2.$$
Therefore, $ |S'|+2= |\mathcal{C}_{G'-S'}|$. Denote $S=S'\cup V(e_{1})\cup V(e_{2})$. Then $ |\mathcal{C}_{G-S}|= |S|-2$ and all components of $G-S$ are factor-critical. 

\vskip 0.1cm

\smallskip

 \noindent {\bf{Claim 1.}}  There is at least one
factor-critical component which is not a singleton, denoted one by $H_{1}$. 
\vskip 0.1cm

\smallskip

Suppose by the contrary that all the factor-critical components are
singletons. Then since $G$ is cubic, by substituting $k=3$ in Theorem \ref{structure}, we obtain that $G$ is isomorphic to $Z_{4n}(1,4n-1,2n)$ or
$Z_{4n+2}(2,4n,2n+1)$ or the Petersen graph, a contradiction.

\vskip 0.1cm

\smallskip

 \noindent {\bf{Claim 2.}}  All the other factor-components of $G-S$ except $H_{1}$ are singletons,  $\partial(H_{1})$ is a cyclic 5-edge-cut and $g\leq 5$.

\vskip 0.1cm

\smallskip

Since $G$ does not contain triangles and $|\overline{V(H_{1})}|\geq 5$,  $d(H_{1})\geq 4$ by Theorem \ref{Xu2000}. Moreover, because $G$ is cubic and $H_{1}$ is an odd component, $d(H_{1})$ and $|H_{1}|$ have the same parity. Hence $d(H_{1})\geq 5$.   $S$ sends at most $3|S|-4$  edges to the factor-critical components; On the other hand, by the 3-edge-connectivity of $G$, the factor-critical components  need at least $5+3(|S|-3)=3|S|-4$ edges from $S$. So there are exactly $3|S|-4$ edges between $S$ and $\overline{S}$. Therefore, $d(H_{1})=5$, all the other components send out exactly three edges and $G[S]$ contains exactly two edges $e_{1}$ and $e_{2}$. Since $G$ does not contain triangles, by Theorem \ref{superlambda}, all the components except $H_{1}$ are singletons. 

Since $H_{1}$ is a factor-critical graph with at least three vertices, $H_{1}$ contains a cycle by Lemma \ref{O.Favaron}. On the other hand, since  $|\overline{H_{1}}|\geq 5$, 

\begin{align*}
||G[\overline{V(H_{1}})||=&\frac{3|\overline{V(H_{1})}|-d(\overline{V(H_{1}}))}{2}=\frac{3|\overline{V(H_{1})}|-d(H_{1})}{2}\\=&
\frac{3|\overline{V(H_{1})}|-5}{2}\geq |\overline{V(H_{1})}|
\end{align*}
holds. Consequently, $G[\overline{V(H_{1}})]$ contains a cycle.  Therefore, $\partial(H_{1})$ is a cyclic edge-cut and further a cyclic 5-edge-cut. Furthermore, by Theorem \ref{Nedela95}, $g=c\lambda(G) \leq d(H_{1})=5$.

\vskip 0.1cm

\smallskip

 If $g=5$, then by Theorem \ref{super cyclically 5-edge-connected},  $G$ is super cyclically 5-edge-connected, that is to say, any cyclic 5-edge-cut isolates a
pentagon. Since $\partial(H_{1})$ is a cyclic 5-edge-cut, either $H_{1}$ or $G[\overline{V(H_{1})}]$ is a pentagon. If $G[\overline{V(H_{1})}]$ is a pentagon, then $G[S]$ is isomorphic to a path of length three, a contradiction. Therefore, $H_{1}$ is a pentagon.

\vskip 0.1cm

\smallskip

\noindent {\bf{Claim 3.}} All edges in $G$ are not remove edges.

\vskip 0.1cm

\smallskip

If this claim holds, then $G$ is uniformly cyclically 5-edge-connected.
Further, if $G$ contains a rosette as its subgraph, then by Lemma \ref{dodecahedron}, $G$ is isomorphic to the dodecahedron. By Theorem \ref{Zhangheping}, it is 2-extendable, contradicting the hypothesis. If $G$ does not contain any rosette as its subgraph, then by Theorem
\ref{Aldred91}, $G$ is either a double ladder or one of $G_{1}$, $G_{2}$.  If $G$ is isomorphic to one of $G_{1}$ and $G_{2}$, then by checking that $G[N_{2}(u)]$ is not isomorphic to $G[N_{2}(v)]$ (see Figure \ref{rossete}), we can see that both $G_{1}$ and $G_{2}$ are not vertex-transitive, a contradiction.
We are left to the case that  $G$ is a double ladder. By Theorems \ref{Aldred91} and \ref{doubleladderextendable}, $G$ is 2-extendable except the Petersen graph, a contradiction too, we are done.

\vskip 0.1cm

\smallskip

\noindent {\em{Proof of Claim 3.}}  Suppose by the contrary that there is an edge $e$  which is removable; that is, $G-e$ is still cyclically 5-edge-connected. Then $e$ cannot be any edge sending out from a pentagon. This is because any set of edges sending out from a pentagon forms a cyclic 5-edge-cut, which can be deduced by a similar way as for $H_{1}$. Consequently, $e$ belongs to some pentagon denoted by $P=w_{1}w_{2}w_{3}w_{4}w_{5}$ and we may suppose that $e=w_{1}w_{2}$. Note that all the edges sending out $P$ are not removable. By the vertex-transitivity of $G$, $w_{3}$ is incident to one edge that is removable and this edge is either $w_{2}w_{3}$ or $w_{3}w_{4}$. We will show that whether what this removable is, all edges in $P$ are removable. If this removable edge incident with $w_{3}$ mentioned above is $w_{2}w_{3}$, then $w_{2}$ is incident two removable edges and further by the vertex-transitivity of $G$, so is $w_{1}, w_{3}, w_{4}$ and $w_{5}$. Therefore, all edges in $P$ are removable. If this removable edge incident with $w_{3}$ is $w_{3}w_{4}$, then also by the vertex-transitivity, $w_{5}$ is incident with one removable edge which is either $w_{4}w_{5}$ or $w_{5}w_{1}$. In either case, we obtain a vertex $w_{4}$ or $w_{1}$ which is incident with two removable edges. By a similar argument as above, we obtain our desire that all edges in $P$ are removable.

Now we know that all edges belonging to some pentagon are removable and the edges not belonging to any pentagon are not removable. It follows that all pentagons are independent. Since $H_{1}$ is a pentagon and all the other pentagons should contain an edge in $G[S]$, there are at most three pentagons. Moreover, since $G$ is of even order, it contains exactly two disjoint pentagons. Similarly as before, $G[\overline{V(H_{1})}]$ can not be a pentagon, a contradiction. 

\vskip 0.1cm

\smallskip

{If $g=4$, then when $G$ contains adjacent qurdangles, we  obtain that $G$ is isomorphic to  $Z_{4n}(1,4n-1,2n)$ or
$Z_{4n+2}(2,4n,2n+1)$  by Lemma \ref{k3g4}, a contradiction. If $G$ is isomorphic to  $T_{m}$ for some integer $m$, then by Lemma \ref{Tmextendable}, it is 2-extendable, a contradiction too. We are left to consider that all quadrangles are independent and $G$ is not isomorphic to $T_{m}$. In this case, it is super cyclically 4-edge-connected by Lemma \ref{g4superaaa}. Since $|H_{1}|$ is odd, there is a quadrangle, denoted by $Q$, containing exactly one or three vertices in $H_{1}$. If $Q$ contains exactly one vertex in $H_{1}$, then this vertex is of degree one in $H_{1}$. But $H_{1}$ is factor-critical, every vertex is of degree at least two  by Lemma \ref{O.Favaron}, a contradiction. If $Q$ contains exactly three vertices of $H_{1}$, then we denote the fourth vertices in $Q$ but not in $H_{1}$ by $v$. The set of edges sending out from $V(H_{1})\cup \{v\}$ forms a cyclic 4-edge-cut of $G$, which follows that $V(H_{1})\cup \{v\}$ or $\overline{V(H_{1})\cup \{v\}}$ induces a quadrangle by the super cyclically 4-edge-connectivity of $G$. If $V(H_{1})\cup \{v\}$ induces a quadrangle, then $H_{1}$ should be a path of length 2, contradicting that $H_{1}$ is factor-critical; If  $\overline{V(H_{1})\cup \{v\}}$ induces a quadrangle, then $G[S]$ contains adjacent edges, contradicting that $G[S]$ contains exactly two edges $e_{1}$ and $e_{2}$. This finally completes the proof. \hspace*{\fill}$\square$





\begin{thebibliography}{20}

\bibitem{Aldred91} R.E.L. Aldred, D.A. Holton and B. Jackson,
Uniform cyclic edge connectivity in cubic graphs, Combinatorica {\bf 11(2)} (1991), 81-96.

\bibitem{Annexstein90} F. Annexstein, M. Baumlag and A.L. Rosenberg,
Group action graphs and parallel architectures, SIAM J. Comput. {\bf 19(3)} (1990), 544-569.

\bibitem{Bai13} B. Bai, X. Huan and Q. Yu, On classification of extendability of Cayley graphs on Dicyclic groups, preprint.

\bibitem{Biggs} N. Biggs, Algebraic Graph Theory, Cambridge University Press, 1993.

\bibitem{Brunat99} J.M. Brunat, M. Espona, M.A. Fiol and O. Serra,
Cayley digraphs from complete generalized cycles, European J.
Combin. {\bf 20} (1999), 337-349.

\bibitem{Carlsson85} G.E. Carlsson, J.E. Cruthirds, H.B. Sexton and C.G.
Wright, Interconnection networks hased on a generalization of
cube-connected cycles, IEEE Tran. Comput. {\bf 34(8)} (1985), 769-772.

\bibitem{Chan95} O. Chan, C. Chen and Q. Yu, On 2-extendable abelian
Cayley graphs, Discrete Math. {\bf 146} (1995), 19-32.

\bibitem{Chen92} C. Chen, On the classification of 2-extendable
Cayley graphs on Dihedral groups, Australas. J. Combin. {\bf 6} (1992),
209-219.

\bibitem{Dean92} N. Dean, The matching extendability of
surfaces, J. Combin. Theory Ser. B {\bf 54} (1992), 133-141.

\bibitem{Diestel06} R. Diestel, Graph Theroy, Springer, New York, 2006.



\bibitem{O.Favaron} O. Favaron, On $k$-factor-critical graphs, Discuss. Math.
Graph Theory {\bf 16} (1996), 41-51.

\bibitem{Xinggao} X. Gao, Q. Li, J. Wang and W. Wang, On 2-Extendable Quasi-abelian Cayley Graphs, Bull. Malays. Math. Sci. Soc. {\bf 38} (2015), 1-15.

\bibitem{Godsil} C. Godsil and G. Royle, Algebraic Graph Theory, Springer, New York, 2001.

\bibitem{Hellwig} A. Hellwig and L. Volkmann, Sufficient conditions for graphs to be $\lambda'$-optimal, super-edge-connected, and maximally edge-connected, J. Graph Theory {\bf 48} (2005), 228-246.

\bibitem{Lou91} D.A. Holton, D. Lou and M.D. Plummer, On the 2-extendability of
planar graphs, Discrete Math. {\bf 96} (1991), 81-89.


\bibitem{Konstantinova08} E. Konstantinova, Some problems on Cayley graphs, Linear Algebra Appl. {\bf 429} (2008), 2754-2769.

\bibitem{li2014} Q. Li, J. He and H. Zhang, Matching preclusion for
vertex-transitive networks, Discrete Appl. Math. {\bf 207} (2016), 90-98.



\bibitem{LovaszandPlummer86} L. Lov\'{a}sz and M.D. Plummer, Matching Theory, Ann. Discrete
Math., Vol. 29, North-Holland, Amsterdam, 1986.

\bibitem{mader} W. Mader, Minimale $n$-fach kantenzusammenhängenden Graphen, Math. Ann. {\bf 191} (1971), 21-28.
\bibitem{meng} J. Meng, Optimally super-edge-connected transitive graphs, Discrete Math. {\bf 260} (2003), 239-248.

\bibitem{Mik09} \v{S}. Miklavi\v{c} and P. \v{S}parl, On
extendability of Cayley graphs, Filomat {\bf 3(23)} (2009), 93-101.

\bibitem{Nedela95} R. Nedela and M. \v{S}koviera, Atoms of cyclic
connectivity in cubic graphs, Math. Slovaca {\bf 45} (1995), 481-499.


\bibitem{Plummer86} M.D. Plummer, Matching extension in bipartite graphs, Utilitas Math. {\bf 54} (1986), 245-258.

\bibitem{Plummer80} M.D. Plummer, On $n$-extendable graphs, Discrete
Math. {\bf 31} (1980), 201-210.








\bibitem{Schrag} G. Schrag and L. Cammack, Some full and fully deterministic generalized Petersen graphs, preprint.



\bibitem{sun} W. Sun and H. Zhang, 4-Factor-criticality of vertex-transitive graphs, The Electronic Journal of Combinatorics, {\bf 23(3)} (2016), \#P3.1.


\bibitem{Tindell96} R. Tindell, Connectivity of Cayley graphs, in: D.Z. Du, D.F. Hsu (Eds.). Combinatorial
Network Theory, Kluwer, Dordrecht, 1996, pp. 41-64.

\bibitem{Tutte47} W.T. Tutte, The facotrization of linear graphs, J. London Math. Soc. 22 (1947) 107-111.




\bibitem{wang} B. Wang and Z. Zhang, On cyclic edge-connectivity of
transitive graphs, Discrete Math. {\bf 309} (2009), 4555-4563.

\bibitem{Wang04} Y. Wang, Super restricted edge-connectivity of
vertex-transitive graphs, Discrete Math. {\bf 289} (2004), 199-205.

\bibitem{Xujunming2000} J. Xu, Restricted edge-connectivity of
vertex-transitive graphs, Chinese Ann. Math. Ser. A {\bf 21} (2000),
605-608.

\bibitem{Dong08}D. Ye and H. Zhang, 2-extendability of toroidal
polyhexes and Klein-bottle polyhexes, Discrete Appl. Math. {\bf 157}
(2009), 292-299.


\bibitem{Zhangheping} H. Zhang and F. Zhang, New lower bound on the number of perfect matchings in fullerene graphs, J. Math. Chem.  {\bf 30(3)} (2001), 343-347.

\bibitem{Zhangzhao2011} Z. Zhang and B. Zhang, Super cyclically
edge connected transitive graphs, J. Comb. Optm. {\bf 22} (2011), 549-562.


\end{thebibliography}
\end{document}